\documentclass[a4paper,reqno,12pt]{amsart}
\usepackage{amsmath, amssymb, latexsym, amscd, amsthm,amsfonts,amstext}
\usepackage[mathscr]{eucal}

\makeatletter
\@namedef{subjclassname@2010}{%
  \textup{2010} Mathematics Subject Classification}
\makeatother
\usepackage{amsmath, amssymb, latexsym, amscd, amsthm,amsfonts,amstext}
\usepackage[unicode]{hyperref}%có màu cho links

\def\leq{\leqslant}
\def\geq{\geqslant}

\def\leq{\leqslant}
\def\geq{\geqslant}

\theoremstyle{definition}

\def\0{\rm \bf 0}

\def\geq{\geqslant}

\numberwithin{equation}{section}

\newtheorem{theorem}{Theorem}[section]
\newtheorem{remark}[theorem]{Remark}

\newtheorem{example}[theorem]{Example}

\begin{document}
\baselineskip=17pt

\markboth{Hoang Nhat Quy}{The complex Monge-Amp\`ere equation in the Cegrell's classes}

\title{The complex Monge-Amp\`ere equation in the Cegrell's classes}

\author{Hoang Nhat Quy}

\address{Faculty of Mathematics\\ The University of Danang - University of Science and Education\\ Da Nang, Viet Nam \\
hoangnhatquy@gmail.com}

\maketitle

\begin{abstract}
In this paper, we give some precise characterizations of existence of solution to the complex Monge - Amp\`ere equation in the classes $\mathcal E_\chi(\Omega)$ and $\mathcal E_{\chi,loc}(\Omega)$.
\end{abstract}



\section{Introduction}\label{sec1}
In this paper we examines the complex Monge - Amp\`ere equation $(dd^cu)^n=\mu$, where $\mu$ is a given non-negative Radon measure and $(dd^c.)^n$ denotes the complex Monger - Amp\`ere operator. If $\mu$ puts mass on a pluripolar set, the solution to $(dd^cu)^n=\mu$ generally cannot be uniquely determined (\cite{Ze97}). Therefore, the existence of solutions to this equation is the main interest of this article. 

Now, we will review some results related to this equation. The first is in \cite{BT1} where the authors found plurisubharmonic solutions of the Dirichlet problem for the complex Monge - Amp\`ere equation with continuous boundary values in a strictly pseudoconvex domain. In \cite{Lem81} and \cite{Lem83}, Lempert showed  the existence of solutions when the support of the given measure is a single point. He also delved into solutions with real-analytic boundary data and logarithmic singularity near the support of the measure. The underlying domain was assumed to be a strictly convex domain in $\mathbb C^n$, as in \cite{CePo97}, the authors studied the Monge - Amp\`ere equation with the Dirac measure as a given measure. 

In the context that $\Omega$ is a hyperconvex domain in $\mathbb C^n$ (see Section 2 for the definition of hyperconvex domain), Demailly proved (Theorem 4.3 in \cite{De87}) that $(dd^cg_{A_1})^n=(2\pi)^n\delta_z$ on a hyperconvex domain $\Omega$, where $\delta_z$ is the Dirac measure at $z$, and $g_{A_1}$ is the pluricomplex Green function with a pole set containing a single point $A_1=\{z\}$. Next, in \cite{Le89}, Lelong introduced the pluricomplex Green function with a finite pole set, $A_k=\{z_1, ..., z_k\}$, and with positive weights $v_1,..., v_k\; (v_j>0, j=1,...,k)$, and proved that $(dd^cg_{A_k})^n=(2\pi)^n\sum_{j=1}^kv_j^n\delta_{z_j}$. Therefore the pluricomplex Green function is not a solution to the complex Monge - Amp\`ere equation if we want the solution to have other boundary values than those which are identically zero. Given a discrete measure with compact support in a hyperconvex domain $\Omega$, Zeriahi proved in \cite{Ze97} that the complex Monge - Amp\`ere equation is solvable for certain continuous boundary values. In \cite{Xi99}, Xing generalized Zeriahi's result in the case where the given boundary values are identically zero. He considered measures that were majorized by the sum of a linear combination of countable numbers of Dirac measures with compact support and a certain regular Monge - Amp\`ere measure. Recently, in \cite{ACCH09} and \cite{Be15}, the author study the complex Monge - Amp\`ere equation on Cegrell's classes. In \cite{ACCH09}, the authors proved that if a non-negative Borel measure is dominated by a complex Monge - Amp\`ere measure, it is a complex Monge - Amp\`ere measure. Also, in \cite{Be15}, Benelkouchi solved the complex Monge - Amp\`ere equation on the weighted pluricomplex energy classes with the maximal plurisubharmonic boundary value given.

In this article, we will discuss some precise characterizations of existence of solution to the complex Monge - Amp\`ere equation in the classes $\mathcal E_\chi(\Omega)$ and $\mathcal E_{\chi, loc}(\Omega)$. We intend to prove that if the non-negative Radon measure is the Monge - Amp\`ere measure localy in $\Omega$, it is the Monge - Amp\`ere measure on $\Omega$.

\section{Preliminaries}\label{sec2}
Throughout this paper it is always assumed that $\Omega$ is a hyperconvex domain, that is, a connected, bounded open subset of $\mathbb C^n$, such that there exists a negative plurisubharmonic function $\rho$ with $\{z\in\Omega: \rho(z)<c\}\Subset\Omega$, for all $c<0$. Such a function $\rho$ is called an exhaustion function. We let PSH$(\Omega)$ denote the cone of plurisubharmonic functions on $\Omega$ and let PSH$^-(\Omega)$ denote the subclass of negative functions. Now, we will recall some important pluricomplex energy classes and measure classes which are used in the next section.

Firstly, we recall some pluricomplex energy classes that was introduced in \cite{Ce98} and \cite{Ce04}:
$$
\begin{aligned}
	&\mathcal E_0(\Omega)=\{\varphi\in\text{L}^\infty(\Omega): \lim_{z\to\xi}\varphi(z)=0,\;\forall\xi\in\partial\Omega \text{ and } \int_\Omega(dd^c\varphi)^n<+\infty\},\\
	&\mathcal F(\Omega)=\left\{\varphi\in\text{PSH}^-(\Omega): \exists\mathcal E_0(\Omega)\ni\varphi_j\searrow\varphi, \sup_{j\ge 1}\int_\Omega(dd^c\varphi_j)^n<+\infty\right\},\\
	&\begin{aligned}
		\mathcal E(\Omega)=\left\{\varphi\in\text{PSH}^-(\Omega):\right.& \forall  z_0\in\Omega,\exists \text{ a neighbourhood } \omega\ni z_0,\\
		& \left. \mathcal E_0(\Omega)\ni\varphi_j\searrow\varphi\text{ on } \omega, \sup_{j\ge 1}\int_\Omega (dd^c\varphi_j)^n<+\infty\right\}.
	\end{aligned}
\end{aligned}
$$

The following weighted pluricomplex energy classes are introduced in \cite{BGZ09} and in \cite{HHQ13}:
$$
\begin{aligned}
	&\mathcal E_\chi(\Omega)=\left\{\varphi\in PSH^-(\Omega): \exists \mathcal E_0(\Omega)\ni\varphi_j\downarrow\varphi, \sup_{j\geq 1}\int_\Omega\chi(\varphi_j)(dd^c\varphi_j)^n<+\infty\right\},\\
	&\mathcal{E}_{\chi, loc} (\Omega)=\{\varphi\in \text{PSH}^-(\Omega):\exists\ \psi_D \in \mathcal{E}_\chi (\Omega) \text{ such that } \varphi = \psi_D\text{ on } D,\ \forall\ D\Subset\Omega\}.
\end{aligned}
$$

\begin{remark}\label{rem24}
	i) By the Remark following Definition 4.6 in \cite{Ce04} we have: for every $u\in\mathcal E(\Omega)$ and every $\omega\Subset\Omega$, there is a $u_\omega\in\mathcal F(\Omega)$ sucth that $u=u_\omega$ on $\omega$.\\
	ii) It is known from \cite{Be09} and \cite{BGZ09} that if $\chi$ is bounded and $\chi(0)>0$ then $\mathcal E_\chi(\Omega)=\mathcal F(\Omega)$.\\
	iii) By Theorem 3.1 in \cite{HH11}, if $\chi\not\equiv 0$ then $\mathcal E_\chi(\Omega)\subset\mathcal E(\Omega)$.\\
	iv) By Remark 3.4 in \cite{HH11}, if $\chi(-\infty)=+\infty$ then $\mathcal E_\chi(\Omega)\subset\mathcal E^a(\Omega)$, where
	$$\mathcal E^a(\Omega):=\{\varphi\in\mathcal E(\Omega): (dd^c\varphi)^n(E)=0,\;\text{$\forall E\subset\Omega$ is the pluripolar set}\}.$$
\end{remark}
We set
$$\mathcal A_M=\{\chi: \mathbb R^-\to\mathbb R^+: \chi \text{ is decreasing and satisfies } \chi(2t)\leq M\chi(t)\}.$$

Throughout this article, we assume that $\chi\in\mathcal A_M$ for some $M>0$.

Then, let $\mathcal M(\Omega)$ be  the set of all non-negative Radon measures on $\Omega$. We set
$$
\begin{aligned}
	&\quad\quad\mathcal M^a(\Omega)=\{\mu\in\mathcal M(\Omega): \mu(E)=0 \text{ for all the Borel pluripolar sets in } \Omega\},\\
	&\quad\quad\mathcal M_\chi(\Omega)=\{\mu\in\mathcal M(\Omega): \exists \varphi\in\mathcal E_\chi(\Omega) \text{ such that } \mu=(dd^c\varphi)^n\},\\
	&\quad\quad\mathcal M_{\chi, loc}(\Omega)=\{\mu\in\mathcal M(\Omega): \exists\varphi\in\mathcal E_{\chi,loc}(\Omega) \text{ such that } \mu=(dd^c\varphi)^n\},\\
	&\quad\quad\mathcal M^\omega_{\chi,loc}(\Omega)=\{\mu\in\mathcal M(\Omega): \mu|_K\in\mathcal M_{\chi,loc}(\Omega) \text{ for all } K\Subset\Omega\}.
\end{aligned}
$$

\begin{remark}\label{rem41}
	i) We can check that $\mathcal M_\chi(\Omega)\subset \mathcal M_{\chi,loc}(\Omega)\subset \mathcal M^\omega_{\chi,loc}(\Omega)$. From Theorem 3.8 in \cite{Qu23} and the main result in \cite{HHQ13}, $\mathcal E_\chi(\Omega)\varsubsetneq\mathcal E_{\chi,loc}(\Omega)$. This combined with the comparison principle, it is inferred that $\mathcal M_\chi(\Omega)\varsubsetneq\mathcal M_{\chi,loc}(\Omega)$. And in the Example \ref{ex25}, we assert that $\mathcal M_{\chi,loc}(\Omega)\varsubsetneq \mathcal M^\omega_{\chi,loc}(\Omega)$.
	
	ii) We have another characterization for the class $\mathcal M^\omega_{\chi,loc}(\Omega)$ as follows:
	$$\mathcal M^\omega_{\chi,loc}(\Omega)=\{\mu\in\mathcal M(\Omega): \mu|_K\in\mathcal M_\chi(\Omega) \text{ for all } K\Subset\Omega\}.$$
	\begin{proof}
		From Theorem 3.8 in \cite{Qu23}, we have
		$$\mathcal M^\omega_{\chi,loc}(\Omega)\supset\{\mu\in\mathcal M(\Omega): \mu|_K\in\mathcal M_\chi(\Omega) \text{ for all } K\Subset\Omega\}.$$
		To prove the converse, we take $\mu\in\mathcal M^\omega_{\chi,loc}(\Omega)$ and $K\Subset\Omega$. Then there exists $\varphi\in\mathcal E_{\chi,loc}(\Omega)$ such that $\mu|_K=(dd^c\varphi)^n$. Let the subset $D$ such that $K\Subset D\Subset \Omega$. Set
		$$h^{\varphi}_{D,\Omega}=\sup\{\psi\in\text{PSH}^-(\Omega): \psi\leq\varphi \text{ on } D\}.$$
		Then $h^{\varphi}_{D,\Omega}\in\mathcal E_\chi(\Omega)$ and 
		$$(dd^ch^{\varphi}_{D,\Omega})^n\geq 1_D(dd^c\varphi)^n\geq \mu|_K.$$
		From Theorem 4.14 in \cite{ACCH09}, we can find $u\in\mathcal E(\Omega)$ such that $u\geq h^{\varphi}_{D,\Omega}$ and $\mu|_K=(dd^cu)^n$. Since $u\in\mathcal E_\chi(\Omega)$ we deduce $\mu|_K\in\mathcal M_\chi(\Omega)$.
	\end{proof}
	
	iii) From Theorem 4.14 in \cite{ACCH09}, we have some more characterizations for the classes $M_\chi(\Omega), M_{\chi, loc}(\Omega)$ and $M^\omega_{\chi,loc}(\Omega)$ as follows:
	$$
	\begin{aligned}
		&\quad\quad\mathcal M_\chi(\Omega)=\{\mu\in\mathcal M(\Omega): \exists\varphi\in\mathcal E_\chi(\Omega) \text{ such that } \mu\leq (dd^c\varphi)^n\},\\
		&\quad\quad\mathcal M_{\chi, loc}(\Omega)=\{\mu\in\mathcal M(\Omega): \exists\varphi\in\mathcal E_{\chi,loc}(\Omega) \text{ such that } \mu\leq  (dd^c\varphi)^n\},\\
		&\quad\quad\mathcal M^\omega_{\chi,loc}(\Omega)=\{\mu\in\mathcal M(\Omega): \forall K\Subset\Omega, \exists\varphi_K\in\mathcal E_{\chi,loc}(\Omega)\\
		&\quad\quad\quad\quad\quad\quad\quad\quad\quad\quad\quad\quad\quad\quad\quad \text{ such that } \mu|_K\leq (dd^c\varphi_K)^n\}.
	\end{aligned}
	$$
\end{remark}

\begin{example}\label{ex25}
	We can find a measure $\mu\in\mathcal M^\omega_{\chi,loc}(\Omega)\setminus\mathcal M_{\chi,loc}(\Omega)$. 
	\begin{proof}
		Let $\Omega_1\Subset\Omega_2\Subset...\Subset\Omega$ and $\Omega=\cup_{j=1}^\infty\Omega_j$. Set
		$$h_{\Omega_j,\Omega}=\sup\{\psi\in\text{PSH}^-(\Omega): \psi\leq -1\;\text{ on }\; \Omega_j\}.$$
		Then $\text{supp}(dd^ch_{\Omega_j,\Omega})\subset \partial\Omega_j$. Set
		$$\mu=\sum_{j=1}^\infty(dd^ch_{\Omega_j,\Omega})^n.$$
		Then we have
		$$
		\begin{aligned}
			\mu|_{\Omega_k}=\sum_{j=1}^k(dd^ch_{\Omega_j,\Omega})^n\leq\left(dd^c\Big(\sum_{j=1}^kh_{\Omega_j,\Omega}\Big)\right)^n.
		\end{aligned}
		$$
		This implies that $\mu|_{\Omega_k}\in\mathcal M_\chi(\Omega),\; \forall k\geq 1$. So $\mu\in\mathcal M^\omega_{\chi,loc}(\Omega)$. \\
		We are going to prove that there does not exist $\varphi\in\mathcal E(\Omega)$ such that $(dd^c\varphi)^n=\mu$. Indeed, we assume that there exists $\varphi\in\mathcal E(\Omega)$ such that $\mu=(dd^c\varphi)^n$. Since
		$$(dd^c\varphi)^n\geq (dd^ch_{\Omega_j,\Omega})^n,\quad\forall j\geq 1$$
		we have $\varphi\leq h_{\Omega_j,\Omega}$ on $\Omega$ for all $j\geq 1$. This infer that $\varphi\leq -1$ on $\Omega$. So we have $u+1\in\text{PSH}^-(\Omega)$ and
		$$\mu=\left(dd^c(\varphi+1)\right)^n.$$
		Repeating the above arguments we get $\varphi\leq -k$ for all $k\geq 1$. This derives $\varphi\equiv-\infty$ on $\Omega$. 
	\end{proof}  	
\end{example}

\section{The complex Monge - Amp\`ere equation in the class $\mathcal E_\chi(\Omega)$}\label{sec4}

In this section, we are going to state some main results on characterizations of existence of solution to the Monge - Amp\`ere equation in the class $\mathcal E_\chi(\Omega)$. 

\begin{theorem}\label{the42}
	Let $\chi\in\mathcal A_M$ and $\mu\in\mathcal M^a(\Omega)$. Then the following two statements are equivalent:\\
	i) There exists $\varphi\in\mathcal E_\chi(\Omega)$ such that $\mu = (dd^c\varphi)^n$,\\
	ii) There exists $A_\mu\in [0,1)$ and $C_\mu>0$ such that 
	$$\int_\Omega\chi(u)d\mu\leq A_\mu e_\chi(u)+C_\mu,\quad \forall u\in\mathcal E_0(\Omega),$$
	where $e_\chi(u)=\int_\Omega\chi(u)(dd^cu)^n$. 
\end{theorem}
\begin{proof}
	$\bullet$ $i)\Rightarrow ii)$: We only need to prove that
	$$\int\limits_{ \Omega } \chi ( u ) ( dd^c  \varphi )^n \leq  A e_\chi ( u ) + B e_\chi ( \varphi ),$$
	where $B = M^2 [ \frac M A ]^{ \frac { \log_2 M } { n } }$. First, we consider the case of $\chi\in C^\infty(\mathbb R)$. We have
	$$\aligned\int\limits_{ \Omega } \chi ( u ) ( dd^c  \varphi )^n& = - \int\limits_{ 0 }^{ + \infty } \chi ' ( - t )  \int\limits_{ \{ u < - t \} } ( dd^c  \varphi )^n dt + \chi ( 0 ) \int\limits_{ \Omega } ( dd^c  \varphi )^n\\
	&\leq I + J + \chi ( 0 ) \int\limits_{ \Omega } ( dd^c  \varphi )^n,
	\endaligned$$
	where
	$$I = - \int\limits_{ 0 }^{ + \infty } \chi ' ( - t )  \int\limits_{ \{ u < \frac { \varphi } { \epsilon } - \frac { t } {2} \} } ( dd^c  \varphi )^n dt,$$
	$$J = - \int\limits_{ 0 }^{ + \infty } \chi ' ( - t )  \int\limits_{ \{ \frac { \varphi } { \epsilon } < - \frac { t } {2} \} } ( dd^c  \varphi )^n dt,$$
	and 
	$$\epsilon = [ \frac { A } { M } ]^{ \frac 1 n }.$$
	From the comparison principle, we have 
	$$\begin{aligned}
		I&=\epsilon^n\int\limits_0^{+\infty}-\chi'(-t)dt\int_{\{u<\frac{\varphi}{\epsilon}-\frac{t}{2}\}}\left(dd^c\left(\frac{\varphi}{\epsilon}-\frac{t}{2}\right)\right)^n
	\end{aligned}$$
	$$\begin{aligned}
		&\leq\epsilon^n\int\limits_0^{+\infty}-\chi'(-t)dt\int_{\{u<\frac{\varphi}{\epsilon}-\frac{t}{2}\}}(dd^cu)^n\\
		&\leq\epsilon^n\int\limits_0^{+\infty}-\chi'(-t)dt\int_{\{u<\frac{-t}{2}\}}(dd^cu)^n\\
		&=\frac{\epsilon^n}{2^n}\int\limits_0^{+\infty}-\chi'(-t)dt\int_{\{2u<-t\}}(dd^c2u)^n\\
		&\leq\frac{\epsilon^n}{2^n}\int\limits_\Omega\chi(2u)(dd^c2u)^n\\
		&=\epsilon^n\int\limits_\Omega\chi(2u)(dd^cu)^n\\
		&\leq\epsilon^nM\int\limits_\Omega\chi(u)(dd^cu)^n=A e_\chi(u),
	\end{aligned}$$
	We have
	$$\begin{aligned}
		J&=\int\limits_0^{+\infty} - \chi'(-t) dt\int\limits_{\{\frac{\varphi}{\epsilon}<\frac{-t}{2}\}}(dd^c\varphi)^n\\
		&=\left(\frac{\epsilon}{2}\right)^n\int\limits_0^{+\infty}-\chi'(-t)dt\int\limits_{ \{ \frac{2\varphi}{\epsilon}<-t \} } \left(dd^c\left(\frac{2\varphi}{\epsilon}\right)\right)^n\\
		&=\left(\frac{\epsilon}{2}\right)^n \left[ \int\limits_\Omega \chi\left(\frac{2\varphi}{\epsilon}\right)\left(dd^c\left(\frac{2\varphi}{\epsilon}\right)\right)^n - \chi(0)\int\limits_\Omega\left(dd^c\frac{2\varphi}{\epsilon}\right)^n\right]\\
		&=\int\limits_\Omega \chi\left(\frac{2\varphi}{\epsilon}\right)(dd^c\varphi)^n-\chi(0)\int\limits_\Omega(dd^c\varphi)^n\\
		&\leq M \left(\frac{2}{\epsilon}\right)^{\log_2M}\int\limits_\Omega \chi(\varphi)(dd^c\varphi)^n-\chi(0)\int\limits_\Omega (dd^c\varphi)^n\\
		&=M^2\left(\frac{1}{\epsilon}\right)^{\log_2M} \int\limits_\Omega \chi(\varphi)(dd^c\varphi)^n - \chi(0)\int\limits_\Omega (dd^c\varphi)^n\\
		&=M^2\left(\frac{M}{A}\right)^{\frac{\log_2M}{n}}\int\limits_\Omega\chi(\varphi)(dd^c\varphi)^n - \chi(0)\int\limits_\Omega (dd^c\varphi)^n
	\end{aligned}$$
	$$\begin{aligned}	
		&=\frac { M^{ 2+\frac {\log_2M} {n} } } { A^{\frac{\log_2M}{n}}}e_\chi(\varphi) - \chi(0)\int\limits_\Omega (dd^c\varphi)^n.
	\end{aligned}$$
	
	Combining these inequalities, we get
	$$\int\limits_{ \Omega } \chi ( u ) ( dd^c  \varphi )^n \leq  A e_\chi ( u ) + B e_\chi ( \varphi ),$$
	In the general case, we set
	$$\chi_{ j } (t) = j \int_0^{\frac 1 j }\chi(t - s)\rho( j s )ds,$$
	where the function $\rho: \mathbb R\to\mathbb R$ is in $C^\infty(\mathbb R)$ such that $\rho(t)\geq 0\; \forall t\in\mathbb R$, $\text{supp}\rho\Subset ( 0, 1)$ and $\int_\mathbb R\rho(t)dt=1$. Then we have $\chi_{j}\in C^\infty(\mathbb R)\cap\mathcal A_M$, $\chi_j\geq\chi$ and $\chi_j\to\chi$ on $\mathbb R^-\backslash \{ t_1, t_2, .... \} $, where $\chi$ is discontinuous on $\{ t_1, t_2, .... \}$. We modify $\chi_j$ at $\{ t_1, t_2, ..., t_j \}$ to construct $\tilde{\chi_j}\in C^\infty(\mathbb R)\cap\mathcal A_M$ such that $\tilde{\chi_j}\geq\chi$ and $\tilde{\chi_j}\to\chi$ on $\mathbb R^-$. From the proof in the first case, we get
	$$\aligned\int\limits_{ \Omega } \chi ( u ) ( dd^c  \varphi )^n &\leq \varliminf\limits_{ j\to +\infty } \int\limits_{ \Omega } \tilde{\chi_j} ( u ) ( dd^c  \varphi )^n\\
	&\leq\varliminf\limits_{ j\to +\infty }A e_{ \tilde{\chi_j} } ( u ) + B e_{ \tilde{\chi_j} } ( \varphi )\\
	&\leq  A e_\chi ( u ) + B e_\chi ( \varphi ),
	\endaligned$$
	
	$\bullet$ $ii)\Rightarrow i)$: From Theorem 6.3 in \cite{Ce98}, we can find a function $\phi\in\mathcal E_0(\Omega)$ and $f\in L^1_{loc}((dd^c\phi)^n)$ such that $\mu=f(dd^c\phi)^n$. \\
	From Theorem A in \cite{Ko95}, there exist $\varphi_j\in\mathcal E_0(\Omega)$ such that 
	$$(dd^c\varphi_j)^n=\min(f,j)(dd^c\phi)^n.$$
	Since $(dd^c\varphi_j)^n\;\nearrow f(dd^c\phi)^n=\mu$ we have $\varphi_j\searrow \varphi\in\text{PSH}^-(\Omega)$. \\
	We have
	$$\int_\Omega\chi(\varphi_j)d\mu\leq A_\mu e_\chi(\varphi_j)+C_\mu.$$
	This implies that
	$$
	\begin{aligned}
		\int_\Omega\chi(\varphi_j)(dd^c\varphi_j)^n\leq \int_\Omega\chi(\varphi_j)d\mu \leq A_\mu e_\chi(\varphi_j)+C_\mu. 
	\end{aligned}
	$$	
	So we have
	$$\int_\Omega\chi(\varphi_j)(dd^c\varphi_j)^n\leq \frac{C_\mu}{1-A_\mu},\quad\forall j\geq 1.$$
	This implies that
	$$\sup_{j\geq 1}\int_\Omega\chi(\varphi_j)(dd^c\varphi_j)^n\leq \frac{C_\mu}{1-A_\mu}<+\infty.$$
	So $\varphi\in\mathcal E_\chi(\Omega)\subset\mathcal E(\Omega)$ and 
	$$(dd^c\varphi)^n=\lim_{j\to\infty}(dd^c\varphi_j)^n\nearrow \mu.$$
\end{proof}

\begin{theorem}\label{the43}
	Let $\chi\in\mathcal A_M$ and $\mu\in\mathcal M(\Omega)$. Then the following two statements are aquivalent:\\
	i) There exists $\varphi\in\mathcal E_\chi(\Omega)$ such that $\mu=(dd^c\varphi)^n$,\\
	ii) There exists $A_\mu\in [0,1), C_\mu>0$ such that
	$$\int_\Omega\chi(u)d\mu\leq A_\mu e_\chi(u)+C_\mu,\quad\forall u\in\mathcal E_0(\Omega)$$
	and for all $z\in\Omega$ there exists a neighbourhood $D$ of $z$ and a function $\varphi_z\in\mathcal E(D)$ such that $\mu\leq (dd^c\varphi_z)^n$.
\end{theorem}
\begin{proof}
	$\bullet$ $i)\Rightarrow ii)$: Using the arguments similar to the proof of Theorem \ref{the42} and $\mathcal E_\chi(\Omega)\subset\mathcal E(\Omega)$.
	
	$\bullet$ $ii)\Rightarrow i)$: From Theorem 6.3 in \cite{Ce98}, we can find a function $\phi\in\mathcal E_0(\Omega)$, $f\in L^1_{loc}((dd^c\phi)^n)$ and $\nu\in\mathcal M(\Omega)$ with $\nu(\Omega\setminus E)=0$ for some Borel pluripolar set $E$ and 
	$$\mu=f(dd^c\phi)^n+\nu.$$
	We consider two cases as follows:\\
	First, we consider the case of $\chi(-\infty)=+\infty$. From Remark 3.4 in \cite{HH11}, we have $\mathcal E_\chi(\Omega)\subset\mathcal E^a(\Omega)$. \\
	Now, we will show that $\mu\in\mathcal M^a(\Omega)$. Indeed: take a pluripolar set $E$ in $\Omega$. From Theorem 3.5 in \cite{Qu23}, there exists a function $\varphi_0\in\mathcal E_\chi(\Omega)$ such that $E= \{\varphi_0=-\infty\}$. \\
	We have
	$$\int_\Omega\chi(\varphi_0)d\mu\leq A_\mu e_\chi(\varphi_0)+C_\mu<+\infty.$$
	This implies that
	$$
	\begin{aligned}
		\mu(E)&\leq \mu(\{\varphi_0<-j\})\leq \frac{1}{\chi(-j)}\int_{\{\varphi_0<-j\}}\chi(\varphi_0)d\mu \\
		&\leq \frac{1}{\chi(-j)}\int_\Omega\chi(\varphi_0)d\mu\leq\frac{A_\mu e_\chi(\varphi_0)+C_\mu}{\chi(-j)},\quad \forall j\geq 1.
	\end{aligned}
	$$
	Letting $j\to\infty$, we get $\mu(E)=0$. Using the Theorem \ref{the42}, we can find $\varphi\in\mathcal E_\chi(\Omega)$ such that $\mu=(dd^c\varphi)^n$.\\
	Second, we consider the case of $\chi(-\infty)<+\infty$. We will prove that there exists $g\in\mathcal E_\chi(\Omega)$ such that
	$$(dd^cg)^n=f(dd^c\phi)^n.$$
	Indeed: from $f(dd^c\phi)^n\in\mathcal M^a(\Omega)$ and 
	$$\int_\Omega\chi(u)f(dd^c\phi)^n\leq\int_\Omega\chi(u)d\mu\leq A_\mu e_\chi(u)+C_\mu,\quad\forall u\in\mathcal E_0(\Omega)$$
	and from Theorem \ref{the42} we can find $g\in\mathcal E^a_\chi(\Omega)$ such that $(dd^cg)^n=f(dd^c\phi)^n$.\\
	Now, we can prove that there exists $h\in\mathcal E_\chi(\Omega)$ such that $(dd^ch)^n=\nu$. Indeed: for all $z\in\Omega$, we can find a neighbourhood $D_z$ of $z$ and $\varphi_z\in\mathcal E(D_z)$ such that 
	$$\nu|_{D_z}\leq \mu|_{D_z}\leq (dd^c\varphi_z)^n.$$
	Take a neighbourhood $V_z$ of $z$ such that $V_z\Subset D_z$. Set
	$$\psi_z=\sup\{\psi\in\text{PSH}^-(D_z): \psi\leq \varphi_z \text{ on } V_z\}.$$
	Then $\psi_z\in\mathcal F(D_z)$ and $\nu|_{V_z}\leq (dd^c\psi_z)^n$.
	We have
	$$
	\begin{aligned}
		\left[\int_\Omega\biggl(dd^c\Big(\sum_{j=1}^k\epsilon_j\widetilde{\psi}_{z_j}\Big)\biggl)^n\right]^{1/n}&\leq\sum_{j=1}^k\left(\int_\Omega \Big(dd^c(\epsilon_j\widetilde{\psi}_{z_j})\Big)^n\right)^{1/n}\\
		&=\sum_{j=1}^k\frac{1}{2^j}\leq 1.
	\end{aligned}
	$$
	This implies that
	$$\sum_{j=1}^\infty\epsilon_j\widetilde{\psi}_{z_j}=:\widetilde\psi\in\mathcal F(\Omega).$$
	We have
	$$\nu|_{V_{z_j}}\leq\frac{1}{\epsilon^n_j}(dd^c\widetilde{\psi})^n,\quad \forall j\geq 1.$$
	From decompotion theorem of measures, we can find $\theta\in L^\infty_{loc}(\Omega)$ such that $\nu=\theta(dd^c\widetilde\psi)^n$. From Theorem 4.8 in \cite{ACCH09}, we can find $h_j\in\mathcal F(\Omega)$ such that $h_j\geq h_{j+1}$ and
	$$(dd^ch_j)^n=\min(\theta, j)(dd^c\widetilde\psi)^n\nearrow \nu.$$
	Set $h=lim_{j\to\infty}h_j\in\text{PSH}^-(\Omega)$. \\
	We have
	$$\int_\Omega\chi(h_j)d\nu\leq\int_\Omega\chi(h_j)d\mu\leq A_\mu e_\chi(h_j)+C_\mu.$$
	So
	$$
	\begin{aligned}
		e_\chi(h_j)&=\int_\Omega\chi(h_j)(dd^ch_j)^n\leq\int_\Omega\chi(h_j)\theta(dd^c\widetilde\psi)^n\\
		&\leq \int_\Omega\chi(h_j)d\nu\leq\int_\Omega\chi(h_j)d\mu \\
		&\leq A_\mu e_\chi(h_j)+C_\mu.
	\end{aligned}
	$$
	This implies that
	$$e_\chi(h_j)\leq\frac{C_\mu}{1-A_\mu},\quad \forall j\geq 1.$$
	Since $\nu(\Omega\setminus E)=0$, we have
	$$
	\begin{aligned}
		(dd^ch_j)^n(\Omega\setminus E)=0&\Rightarrow\int_{\{h_j>-\infty\}}(dd^ch_j)^n=0\\
		&\Rightarrow(dd^ch_j)^n=1_{\{h_j=-\infty\}}(dd^ch_j)^n
	\end{aligned}
	$$
	We also have
	$$
	\begin{aligned}
		e_\chi(h_j)&=\int_\Omega\chi(h_j)(dd^ch_j)^n\\
		&=\int_{\{h_j=-\infty\}}\chi(h_j)(dd^ch_j)^n\\
		&=\chi(-\infty)\int_\Omega(dd^ch_j)^n.
	\end{aligned}
	$$
	Combining these inequality, we have
	$$
	\begin{aligned}
		\sup_{j\geq 1}\int_\Omega(dd^ch_j)^n\leq\sup_{j\geq 1}\frac{e_\chi(h_j)}{\chi(-\infty)}\leq\frac{C_\mu}{(1-A_\mu)\chi(-\infty)}<+\infty.
	\end{aligned}
	$$ 
	Therefore $h_j\searrow h\in\mathcal F(\Omega)$ and 
	$$(dd^ch)^n=\lim_{j\to\infty}(dd^ch_j)^n=\nu.$$
	Finally, we have
	$$\mu=f(dd^c\phi)^n+\nu\leq (dd^cg)^n+(dd^ch)^n\leq (dd^c(g+h))^n.$$
	From Theorem 4.14 in \cite{ACCH09}, we can find a function $\varphi\in\mathcal E(\Omega)$ such that $\varphi\geq g+h$ and $\mu=(dd^c\varphi)^n$. \\
	Since $\chi(-\infty)<+\infty$, we have $h\in\mathcal F(\Omega)\subset\mathcal E_\chi(\Omega)$. This implies that $\varphi\in\mathcal E_\chi(\Omega)$. 
\end{proof}

\begin{theorem}\label{the44}
	Let $\mu$ be a non-negative Radon measure on $\Omega$. Then the following two statements are equivalent:\\
	i) There exists $\varphi\in\mathcal E_\chi(\Omega)$ such that $\mu=(dd^c\varphi)^n$,\\
	ii) For all $z\in\overline\Omega$, there exists a neighbourhood $D$ of $z$ and $\psi\in\mathcal E_\chi(\Omega\cap D)$ such that $\mu|_{\Omega\cap D}=(dd^c\psi)^n$. 
\end{theorem}
\begin{proof}
	$\bullet$ $i)\Rightarrow ii)$: Take $\varphi\in\mathcal E_\chi(\Omega)$ such that $\mu=(dd^c\varphi)^n$. Take $z\in\overline{\Omega}$ and a neighbourhood $D$ of $z$. We consider two cases:\\
	First, in the case of $\chi(-\infty)=+\infty$. From Remark 3.4 in \cite{HH11}, we have $\mu(E)=0$ for all Borel pluripolar set $E\subset \Omega$. From Theorem 6.3 in \cite{Ce98}, we can find $\phi\in\mathcal E_0(\Omega\cap D), f\in L^1_{loc}((dd^c\phi)^n)$ such that
	$$\mu=f(dd^c\phi)^n\quad\text{ on } \Omega\cap D.$$
	From Theorem A in \cite{Ko98}, there exists $\psi_j\in\mathcal E_0(\Omega\cap D)$ such that 
	$$(dd^c\psi_j)^n=\min(f,j)(dd^c\phi)^n.$$
	Basing on the comparison principle, we have $\psi_j\searrow\psi\geq \varphi|_{\Omega\cap D}$ as $j\to\infty$ and
	$$(dd^c\psi)^n=\lim_{j\to\infty}(dd^c\psi_j)^n=\mu|_{\Omega\cap D}.$$
	We have
	$$\sup_{j\geq 1}\int_{\Omega\cap D}\chi(\psi_j)(dd^c\psi_j)^n\leq\int_{\Omega\cap D}\chi(\varphi)(dd^c\varphi)^n<+\infty.$$
	This implies that $\psi\in\mathcal E_\chi(\Omega\cap D)$.\\
	Second, in the case of $\chi(-\infty)<+\infty$. We have
	$$
	\begin{aligned}
		\int_{\{\varphi=-\infty\}}(dd^c\varphi)^n&=\frac{1}{\chi(-\infty)}\int_{\{\varphi=-\infty\}}\chi(-\infty)(dd^c\varphi)^n\\
		&\leq \frac{1}{\chi(-\infty)}\int_\Omega \chi(\varphi)(dd^c\varphi)^n<+\infty.
	\end{aligned}
	$$
	From Lemma 4.7 and the proof of Theorem 4.8 in \cite{ACCH09}, we can find $v\in\mathcal F(\Omega\cap D)$ such that
	$$(dd^cv)^n=1_{\{\varphi=-\infty\}\cap D}(dd^c\varphi)^n.$$
	Repeating arguments in the first case, we can find $u_j\in\mathcal E_0(\Omega\cap D)$ such that $u_j\geq\varphi|_{D\cap\Omega}$ and 
	$$(dd^cu_j)^n\nearrow 1_{\{\varphi>-\infty\}\cap D}(dd^c\varphi)^n.$$
	We set	
	$$\psi_j=\sup\{u\in\mathcal E(D\cap\Omega): (dd^cu_j)^n\leq(dd^cu)^n \text{ and } u\leq v\}\geq \varphi.$$
	From Lemma 4.13 in \cite{ACCH09}, we have $u_j+v\leq\psi_j\leq v$ and
	$$(dd^c\psi_j)^n=(dd^cu_j)^n+(dd^cv)^n.$$
	Since $(dd^cu_j)^n$ is increasing sequence we have $\psi_j\searrow\psi\geq \varphi+v$. \\
	Moreover, we have
	$$(dd^c\psi_j)^n\to 1_{\{\varphi>-\infty\}\cap D}(dd^c\varphi)^n+1_{\{\varphi=-\infty\}\cap D}(dd^c\varphi)^n=(dd^c\varphi)^n.$$
	This implies that
	$$(dd^c\psi)^n=1_{\{D\cap \Omega\}}(dd^c\varphi)^n.$$
	We hereby prove that $\psi\in\mathcal E_\chi(D\cap\Omega)$. Since
	$$\sup_{j\geq 1}\int_{D\cap\Omega}\chi(u_j)(dd^cu_j)^n\leq\int_{D\cap\Omega}\chi(\varphi)1_{\varphi>-\infty}(dd^c\varphi)^n<+\infty.$$
	We get $u_j\searrow h\in\mathcal E_\chi(D\cap \Omega)$. Since $\chi(-\infty)<+\infty$ we have $\mathcal F(D\cap\Omega)\subset\mathcal E_\chi(D\cap\Omega)$. This implies that $v\in\mathcal E_\chi(D\cap\Omega)$. \\
	Moreover, we have $h+v\leq\psi\leq v$. Therefore $\psi\in\mathcal E_\chi(D\cap\Omega)$. 
	
	$\bullet$ $ii)\Rightarrow i)$: From Theorem \ref{the43}, we only need to prove that
	$$\int_\Omega\chi(u)d\mu\leq A_\mu e_\chi(u)+C_\mu, \quad \forall u\in\mathcal E_0(\Omega),$$
	for some $A_\mu\in [0,1), C_\mu>0$. \\
	Since $\overline{\Omega}$ is compact, we can choose $D_j\Subset\Omega_j$ and $\psi_j\in\mathcal E_\chi(\Omega_j\cap\Omega)$ such that
	$$\cup_{j=1}^mD_j\supset\overline{\Omega}\quad\text{and}\quad \mu|_{\Omega_j\cap\Omega}=(dd^c\psi_j)^n,\; \forall j=\overline{1,m}.$$
	Set
	$$\varphi_j=\sup\{\psi\in\text{PSH}^-(\Omega_j\cap\Omega): \psi\leq \psi_j\;\text{ on }\; D_j\cap \Omega\}.$$
	Then $\varphi_j\in\mathcal E_\chi(\Omega_j\cap\Omega)$, $\mu|_{D_j\cap\Omega}\leq (dd^c\varphi_j)^n$ and $\text{supp}(dd^c\varphi_j)^n\subset\overline{D}_j\cap\Omega$. 
	Take $u\in\mathcal E_0(\Omega)$. We set
	$$u_j=\sup\{\psi\in\text{PSH}^-(\Omega_j\cap\Omega): \psi\leq u\;\text{ on }\; \overline{D}_j\cap\Omega\}.$$
	Then $u_j\in\mathcal E_0(\Omega_j\cap\Omega)$. From Lemma 3.3 and arguments in the proof of (ii) in Theorem 3.4 in \cite{HHQ13}, we have
	$$
	\begin{aligned}
		\int_{\Omega_j\cap\Omega}\chi(u_j)(dd^cu_j)^n&\leq c_n\int_\Omega\chi(u)(dd^cu)^n+\int_K|u|\chi(u)dV_{2n}\\
		&\leq d_n\int_\Omega\chi(u)(dd^cu)^n
	\end{aligned}
	$$
	for some $c_n>0,\; K\Subset\Omega$.\\
	We have
	$$
	\begin{aligned}
		\int_\Omega\chi(u)d\mu&\leq\sum_{j=1}^m\int_{D_j\cap\Omega}\chi(u)d\mu\\
		&\leq\sum_{j=1}^m\int_{D_j\cap\Omega}\chi(u_j)(dd^c\varphi_j)^n.
	\end{aligned}
	$$
	Basing on Theorem \ref{the43}, we can find a constants
	$$A_j=\min\left(\frac{1}{2md_n}, \frac{1}{2m}\right),\; c_j>0$$
	such that
	$$\int_{\Omega_j\cap\Omega}\chi(v)(dd^c\varphi_j)^n\leq A_je_\chi(v)+c_j,\quad\forall v\in\mathcal E_0(\Omega_j\cap\Omega).$$
	So
	$$
	\begin{aligned}
		\int_\Omega\chi(u)(dd^cu)^n&\leq \sum_{j=1}^mA_je_\chi(u_j)+\sum_{j=1}^m c_j\\
		&\leq \sum_{j=1}^mA_jd_ne_\chi(u)+\sum_{j=1}^mc_j\\
		&<\frac{1}{2}e_\chi(u)+C\quad(\text{where } C=\sum_{j=1}^mc_j). 
	\end{aligned}
	$$
\end{proof}

\begin{theorem}\label{the46}
	Let $\mu$ be a non-negative Radon measure on $\Omega$. Then the following two statements are equivalent:\\
	i) There exists $\varphi\in\mathcal E_{\chi,loc}(\Omega)$ such that $\mu=(dd^c\varphi)^n$,\\
	ii) For all $z\in\overline{\Omega}$ there exists a neighbourhood $D$ of $z$ and $\psi\in\mathcal E_{\chi,loc}(D\cap\Omega)$ such that $\mu|_{D\cap\Omega}=(dd^c\psi)^n$.  
\end{theorem}
\begin{proof}
	$\bullet$ $i)\Rightarrow ii)$: Take $\varphi\in\mathcal E_{\chi,loc}(\Omega)$ such that $\mu=(dd^c\varphi)^n$. Take $z\in\overline{\Omega}$ and a neighbourhood $D$ of $z$. From Theorem 3.4 in \cite{HHQ13}, we have $\varphi|_{D\cap\Omega}\in\mathcal E_{\chi,loc}(D\cap\Omega)$. For all $D'\Subset D\cap\Omega$, we set
	$$h^\varphi_{D',D\cap\Omega}=\sup\{u\in\text{PSH}^-(D\cap\Omega): u\leq\varphi \text{ on } D'\}.$$
	Then $h^\varphi_{D',D\cap\Omega}\in\mathcal E_\chi(D\cap\Omega)$ and
	$$(dd^ch^\varphi_{D',D\cap\Omega})^n\geq1_{D'}(dd^c\varphi)^n=\mu|_{D'}.$$
	This infer that $\mu|_{D'}\in\mathcal M_\chi(D\cap\Omega)$. So there exists $\psi_{D'}\in\mathcal E_\chi(D\cap\Omega)$ such that $\mu|_{D'}=(dd^c\psi_{D'})^n$. \\
	Take $D'_j\Subset D\cap\Omega$ such that $\cup_{j=1}^kD'_j=D\cap\Omega$. Then we have
	$$
	\begin{aligned}
		\mu|_{D\cap\Omega}\leq\sum_{j=1}^k\mu|_{D'_j}=\sum_{j=1}^k(dd^c\psi_{D'_j})^n\leq\left(dd^c\Big(\sum_{j=1}^k\psi_{D'_j}\Big)\right)^n.
	\end{aligned}
	$$
	From Theorem 4.14 in \cite{ACCH09}, we can find a function $\psi\in\mathcal E(D\cap\Omega)$ such that $\psi\geq \sum_{j=1}^k\psi_{D'_j}$ and $\mu|_{D\cap\Omega}=(dd^c\psi)^n$.

	$\bullet$ $ii)\Rightarrow i)$: Take $z\in\overline{\Omega}$, a neighbourhood $D$ of $z$ and $\psi\in\mathcal E_{\chi,loc}(D\cap\Omega)$ such that $\mu|_{D\cap\Omega}=(dd^c\psi)^n$. By repeating arguments similar to the above proof, we can find $\varphi\in\mathcal E_\chi(D\cap\Omega)$ such that $\mu|_{D\cap\Omega}=(dd^c\varphi)^n$. Now, from Theorem \ref{the44} we have i).	
\end{proof}

\end{document}